\documentclass[12pt]{amsart}
\usepackage{a4wide}
\usepackage[utf8]{inputenc}
\usepackage{tikz}
\usepackage[backend=biber,style=alphabetic,maxbibnames=99]{biblatex}  
\usepackage{multirow}
\usepackage{textcomp}
\usepackage{csquotes} 
\usepackage{listings}
\usepackage{hyperref}
\hypersetup{colorlinks=false}
\usepackage{amssymb}
\usepackage{enumitem}
\usepackage[none]{hyphenat}

\usepackage{amsmath,amssymb}
\newcommand{\vertiii}[1]{{\left\vert\kern-0.25ex\left\vert\kern-0.25ex\left\vert #1 
    \right\vert\kern-0.25ex\right\vert\kern-0.25ex\right\vert}}

\newtheorem{theorem}{Theorem}[section]
\newtheorem{thmx}{Theorem}

\newtheorem{proposition}[theorem]{Proposition}
\newtheorem{lemma}[theorem]{Lemma}

\theoremstyle{definition}
\newtheorem{definition}[theorem]{Definition}

\newtheorem{remark}[theorem]{Remark}

\author{Petr H\'ajek}
\address[P. H\'ajek]{Czech Technical University in Prague, FEE, Dept.of mathematics, Technicka 2, 160 00 Prague 6 (Czech Republic)}
\email{hajek@math.cas.cz}

\author{Andr\'es Quilis}
\address[A. Quilis]{Universitat Polit\`ecnica de Val\`encia. Instituto Universitario de Matem\'atica Pura y Aplicada, Camino de Vera, s/n
46022 Valencia (Spain); and Czech Technical University in Prague, FEE, Dept.of mathematics, Technicka 2, 160 00 Prague 6 (Czech Republic)}
\email{anquisan@posgrado.upv.es}
\subjclass[2020]{46B03, 46B10}
\title[Counterexamples in rotundity of norms in Banach spaces]{Counterexamples in rotundity of norms in Banach spaces}
\keywords{Strict convexity, uniformly rotund norms, weakly uniformly rotund norms, uniformly rotund in every direction norms, higher order smoothness}
\begin{document}

\begin{abstract}
    We study several classical concepts in the topic of strict convexity of norms in infinite dimensional Banach spaces. Specifically, and in descending order of strength, we deal with Uniform Rotundity (UR), Weak Uniform Rotundity (WUR) and Uniform Rotundity in Every Direction (URED). Our first three results show that we may distinguish between all of these three properties in every Banach space where such renormings are possible. Specifically, we show that in every infinite dimensional Banach space which admits a WUR (resp. URED) renorming, we can find a norm with the same condition and which moreover fails to be UR (resp. WUR). We prove that these norms can be constructed to be Locally Uniformly Rotund (LUR) in Banach spaces admitting such renormings. Additionally, we obtain that in every Banach space with a LUR norm we can find a LUR renorming which is not URED. These results solve three open problems posed by A.J. Guirao, V. Montesinos and V. Zizler. The norms we construct in this first part are dense. 

    In the last part of this note, we solve a fourth question posed by the same three authors by constructing a $C^\infty$-smooth norm in $c_0$ whose dual norm is not strictly convex. 
\end{abstract}
\maketitle
\section{Introduction}

In this article, we obtain four main results related to convexity and smoothness of renormings in infinite dimensional Banach spaces. These theorems answer four open questions posed in the recently published monograph \cite{GuiMonZiz22}. The first three results deal with different strengthenings of strict convexity (or rotundity). In particular, we work with the following classical concepts:
\begin{definition}
    Let $X$ be a Banach space and let $\|\cdot\|$ be a norm in $X$. Denote by $B_{\|\cdot\|}$ and $S_{\|\cdot\|}$ the unit ball of $X$ and the unit sphere of $X$ respectively in the norm $\|\cdot\|$.
    \begin{itemize}
        \item[-] The norm $\|\cdot\|$ is \emph{uniformly rotund (UR)} if for every pair of sequences $\{x_n\}_{n\in\mathbb{N}}$ and $\{y_n\}_{n\in\mathbb{N}}$ in $B_{\|\cdot\|}$ such that $\left\|\frac{x_n+y_n}{2}\right\|\rightarrow 1$ we have that $\|y_n-x_n\|\rightarrow 0$.
        \item[-] The norm $\|\cdot\|$ is \emph{weakly uniformly rotund (WUR)} if for every pair of sequences $\{x_n\}_{n\in\mathbb{N}}$ and $\{y_n\}_{n\in\mathbb{N}}$ in $B_{\|\cdot\|}$ such that $\left\|\frac{x_n+y_n}{2}\right\|\rightarrow 1$ we have that $y_n-x_n\rightarrow 0$ in the weak topology of $X$.
        \item[-] The norm $\|\cdot\|$ is \emph{uniformly rotund in every direction (URED)} if for every $v\in S_{\|\cdot\|}$ and every pair of sequences $\{x_n\}_{n\in\mathbb{N}}$ and $\{y_n\}_{n\in\mathbb{N}}$ in $B_{\|\cdot\|}$ such that $y_n-x_n=\lambda_nv$ for some $\lambda_n\in\mathbb{N}$ and every $n\in\mathbb{N}$, and such that $\left\|\frac{x_n+y_n}{2}\right\|\rightarrow 1$, we have that $\lambda_n\rightarrow 0$.
        \item[-] The norm $\|\cdot\|$ is \emph{locally uniformly rotund (LUR)} at a point $x\in S_{\|\cdot\|}$ if for every sequence $\{y_n\}_{n\in\mathbb{N}}\in B_{\|\cdot\|}$ such that $\left\|\frac{x+y_n}{2}\right\|\rightarrow 1$ we have that $\|x-y_n\|\rightarrow 0$.
    \end{itemize}
\end{definition}
All four of these concepts have been studied extensively in the field of Renorming of Banach spaces, producing numerous and deep results with equally important applications. We refer to \cite{DevGodZiz93} or \cite{GuiMonZiz22} for a comprehensive view of the topic. 

It is direct from the definition that UR implies WUR, which in turn implies URED. It is also well known that these concepts are distinct in general. For instance, it is a classical result that a Banach space is superreflexive if and only if it admits a UR renorming (see e.g.: Theorem IV.4.1 in \cite{DevGodZiz93}), while the class of Banach spaces admitting a WUR norm is much larger, and for separable Banach spaces it is equivalent to having a separable dual (see \cite{Haj96}). All separable Banach spaces can be renormed with a URED norm (\cite{Ziz71}). 

In this note we show that we may differentiate between all of these notions in every infinite dimensional Banach space in which such renormings are possible. Moreover, all renormings we construct may be taken to be LUR if the Banach space we are dealing with admits a norm with this property; which shows that the three uniform concepts are different even in the presence of strong local convexity properties. Finally, we also prove the density of this kind of norms.

Let us now state precisely the first three main results:
\begin{thmx}
    \label{Main_Theorem_A}
    Let $X$ be an infinite dimensional Banach space with an LUR norm. Then there exists an equivalent norm in $X$ which is LUR and fails to be URED. Moreover, the class of norms with this property is dense.
\end{thmx}
\begin{thmx}
    \label{Main_Theorem_B}
    Let $X$ be an infinite dimensional Banach space with a URED norm. Then there exists an equivalent norm in $X$ which is URED and not WUR. If $X$ admits a LUR norm, then this norm can also be taken to be LUR. Moreover, the class of norms with this property is dense.
\end{thmx}

\begin{thmx}
    \label{Main_Theorem_C}
    Let $X$ be an infinite dimensional superreflexive Banach space. Then there exists an equivalent norm in $X$ which is LUR and WUR but not UR. Moreover, the class of norms with this property is dense.
\end{thmx}

As mentioned above, these three theorems answer three questions in \cite{GuiMonZiz22}, specifically Questions 52.3.4, 52.3.7 and 52.3.1 respectively (page 500). Notice as well that, by duality, Theorem \ref{Main_Theorem_C} implies that in every superreflexive space we may approximate every norm by a Fréchet smooth norm which is Uniformly Gâteaux but fails to be Uniformly Fréchet. This answers Question 52.1.2.4 of \cite{GuiMonZiz22} as well, which was already solved differently in \cite{Qui22} by constructing a Fréchet differentiable norm which fails to be Uniformly Gâteaux. 

The renormings we construct to prove Theorems \ref{Main_Theorem_A}, \ref{Main_Theorem_B} and \ref{Main_Theorem_C} come from a single method, applied with varying parameters to obtain the desired properties in each situation. Intuitively, the way we build these renormings is by defining first a countable family of norms with the following property: their unit sphere coincides with the original unit sphere except in a particular slice, where the new sphere contains a strictly convex approximation of a segment in a given direction, which may be different for each of the countably many norms. By suitably combining the countable family of norms, we obtain a renorming which fails to be UR, WUR or URED while satisfying a weaker rotundity condition, depending on the choice of the directions in which the approximated segments appear. The idea of using countably many norms which differ from the original norm only on a certain slice was also used in \cite{Qui22}, where such a technique is applied to the construction of norms with specific smoothness properties.

To finish the introduction, let us discuss the fourth and last of the main theorems of this article, which we state now:
\begin{thmx}
    \label{Main_Theorem_D}
    There exists a $C^\infty$-smooth norm on $c_0$ such that the dual norm is not strictly convex. Moreover, the norm $\|\cdot\|_\infty$ can be uniformly approximated on bounded sets by norms with this property. 
\end{thmx}

This result answers in particular part (i) of Question 139 in \cite{GuiMonZiz16} (part (ii) was already solved in \cite{Qui22}), which is posed again as Question 52.1.4.6 in \cite{GuiMonZiz22} (page 498). To put it in context, recall that by a classical \v{S}mulyan result, a norm $\|\cdot\|$ in any Banach space is Gâteaux differentiable as soon as the dual norm $\|\cdot\|^*$ is strictly convex. It is known that the converse is not true in general: indeed, by a result of V. Klee \cite{Kle59} (Proposition 3.3), a Gâteaux differentiable norm with non strictly convex dual can be constructed in every separable non-reflexive Banach space. For such spaces which additionally have a separable dual (such as $c_0$), this norm can be taken to be Fréchet differentiable, as shown by A.J. Guirao, V. Montesinos and V. Zizler in \cite{GuiMonZiz12}. Even in reflexive spaces, where every Gâteaux differentiable norm does have a strictly convex dual norm, a classical result from D. Yost in \cite{Yos81} proves that we may construct Fréchet differentiable norms whose dual norm is not LUR.

The construction of the norm in the proof of Theorem \ref{Main_Theorem_D} is based around the $C^\infty$-smooth approximation of certain $n$-dimensional polyhedra constructed inductively.

Let us now briefly discuss the structure of this article. In section 2 we set the notation to be used throughout the rest of the note, and we recall some more definitions and preliminary results. In section 3 we lay out the construction of the renormings in order to prove Theorems \ref{Main_Theorem_A}, \ref{Main_Theorem_B} and \ref{Main_Theorem_C} regarding rotundity. Finally, section 4 is dedicated to proving Theorem \ref{Main_Theorem_D} about a $C^\infty$-smooth norm in $c_0$ with non strictly convex dual norm.

\section{Notation and Preliminary results}

We write $B_{\|\cdot\|}$ and $S_{\|\cdot\|}$ to denote the unit ball and the unit sphere of a Banach space with respect to the norm $\|\cdot\|$.

We use the definitions of UR, WUR, URED and LUR given in the previous section. Additionally, we say that a norm $\|\cdot\|$ in a Banach space $X$ is \emph{strictly convex} if whenever $x,y\in B_{\|\cdot\|}$ with $x\neq y$ we have that $\left\|\frac{x+y}{2}\right\|<1$. 

Regarding smoothness, we will use the standard definitions of differentiability in Banach spaces, which can be found, for instance, in \cite{DevGodZiz93}.

Given a Banach space $X$ and a class of equivalent norms $\mathcal{S}$ defined on $X$, we say that a norm $\|\cdot\|$ is \emph{uniformly approximated on uniform sets} by norms in $\mathcal{S}$ if for every $\varepsilon>0$ there exists a norm $\vertiii{\cdot}\in\mathcal{S}$ such that $\vertiii{\cdot}\leq\|\cdot\|\leq(1+\varepsilon)\vertiii{\cdot}$. If every norm in $X$ can be uniformly approximated on uniforms sets by norms in a class $\mathcal{S}$, we say that the class $\mathcal{S}$ is \emph{dense}.

Following the notation of \cite{GuiMonZiz22}, given a norm $\|\cdot\|$ in a Banach space $X$ and two points $x,y\in X$, we define the expression 
$$Q_{\|\cdot\|}(x,y)=2\|x\|^2+2\|y\|^2-\|x+y\|^2\geq 0. $$
We will use the following remark with an elementary proof regarding this function:
\begin{remark}
\label{remark_Q}
Let $X$ be a Banach space. For every $n\in\mathbb{N}$, let $\|\cdot\|_n$ be an equivalent norm on $X$, and let $\{x_n\}_{n\in\mathbb{N}},  \{y_n\}_{n\in\mathbb{N}}\subset X$ be two sequences such that $\limsup_{n\rightarrow\infty}\|x_n\|_n\leq 1$, $\limsup_{n\rightarrow\infty}\|y_n\|_n\leq 1$, and $\left\|\frac{x_n+y_n}{2}\right\|_n\rightarrow 1$. Then $Q_{\|\cdot\|_n}(x_n,y_n)\rightarrow 0$.  
\end{remark}

As suggested by the previous remark, the function $Q_{\|\cdot\|}$ can be used to describe rotundity qualities of the norm $\|\cdot\|$. In particular, we will use the following known characterizations:

\begin{proposition}
    \label{Prop3.UR_Q_Characterization}
    Let $X$ be a Banach space. A norm $\|\cdot\|$ is UR if and only if for every pair of sequences $\{x_n\}_{n\in\mathbb{N}}$ and $\{y_n\}_{n\in\mathbb{N}}$ in $X$ such that $Q_{\|\cdot\|}(x_n,y_n)\rightarrow 0$, we have that $\|x_n-y_n\|\rightarrow 0$.
\end{proposition}
\begin{proposition}
    \label{Prop4.WUR_Q_Characterization}
    Let $X$ be a Banach space. A norm $\|\cdot\|$ is WUR if and only if for every pair of sequences $\{x_n\}_{n\in\mathbb{N}}$ and $\{y_n\}_{n\in\mathbb{N}}$ in $X$ such that $Q_{\|\cdot\|}(x_n,y_n)\rightarrow 0$, we have that $x_n-y_n\rightarrow 0$ in the weak topology of $X$.
\end{proposition}
\begin{proposition}
    \label{Prop2.URED_Q_Characterization}
    Let $X$ be a Banach space. A norm $\|\cdot\|$ is URED if and only if for every $v\in X\setminus \{0\}$, and every pair of sequences $\{x_n\}_{n\in\mathbb{N}}$ and $\{y_n\}_{n\in\mathbb{N}}$ in $X$ such that $y_n-x_n\in \text{span}(v)$ and $Q_{\|\cdot\|}(x_n,y_n)\rightarrow 0$, we have that $\|x_n-y_n\|\rightarrow 0$.
\end{proposition}
\begin{proposition}
    \label{Prop1.LUR_Q_Characterization}

    Let $X$ be a Banach space. A norm $\|\cdot\|$ is LUR at a point $x\in X$ if and only if for every sequence $\{y_n\}_{n\in\mathbb{N}}$ in $X$ such that $Q_{\|\cdot\|}(x,y_n)\rightarrow 0$ we have that $\|x-y_n\|\rightarrow 0$.
\end{proposition}

We refer to \cite{GuiMonZiz22} again for a thorough study of this function and a comprehensive list of characterizations of rotundity concepts it can be used for. 

In section 3 we will define several renormings of the form  $\left( a_1\|\cdot\|^2_1+a_2\|\cdot\|^2_2\right)^{1/2}$, where $\|\cdot\|_1,\|\cdot\|_2$ are two equivalent norms in a Banach space $X$, and $a_1,a_2$ are positive numbers. These renormings are a well known technique to obtain rotund approximations of a given norm. The key property they enjoy in this regard is the fact that 
\begin{equation}
    \label{LinearityQ}
    Q_{\left( a_1\|\cdot\|^2_1+a_2\|\cdot\|^2_2\right)^{1/2}}(x,y)=a_1Q_{\|\cdot\|_1}(x,y)+a_2Q_{\|\cdot\|_2}(x,y),  
\end{equation}
which, together with the characterizations we allude to previously, shows that the renorming $\left( a_1\|\cdot\|^2_1+a_2\|\cdot\|^2_2\right)^{1/2}$ inherits the strongest rotundity conditions of both $\|\cdot\|_1$ and $\|\cdot\|_2$. More precisely, with these observations we immediately obtain the following lemma:
\begin{lemma}
\label{Lemma5:Averageofrotundisrotund}
Let $X$ be a Banach space, let $\|\cdot\|_1$ and $\|\cdot\|_2$ be two equivalent norms on $X$, and let $a_1,a_2>0$. If $\|\cdot\|_1$ is UR (resp. WUR, URED, LUR), then $\left( a_1\|\cdot\|^2_1+a_2\|\cdot\|^2_2\right)^{1/2}$ is UR (resp. WUR, URED, LUR).
\end{lemma}

Let us finish by stating another result which easily follows from the Propositions above stated, and whose proof we also omit.

\begin{lemma}
    \label{Lemma6:Maximumofrotund}
    Let $X$ be a Banach space, let $\{\|\cdot\|_i\}_{i=1}^n$ be a finite sequence of equivalent norms on $X$. If $\|\cdot\|_i$ is UR (resp. WUR, URED, LUR) for every $i=1,\dots,n$, then the equivalent norm $\max_{i=1,\dots,n}\{\|\cdot\|_i\}$ is UR (resp. WUR, URED, LUR).
\end{lemma}

Notice that in Lemma \ref{Lemma5:Averageofrotundisrotund} we only need rotundity of one of the norms, while in Lemma \ref{Lemma6:Maximumofrotund} it is necessary that all norms share the same property.

\section{Rotundity}

This section is divided into three further subsections. In the first section we define a norm in any Banach space which approximates a non strictly convex norm whose unit sphere contains a specific segment. The approximation is done in such a way that the resulting norm has the strongest rotundity qualities of the starting norm.

\subsection{The norm $\vertiii{\cdot}_{\alpha,\delta,C}$}\hfill\\
 Let $(X,\|\cdot\|)$ be a Banach space. Let $0<\delta<1$ small enough such that $(1-\delta)^4>\frac{1}{2}$, let $C\geq 1$ and let $\alpha=(x_0,h_0,f_0,g_0)$ be a $4$-tuple of elements such that $x_0\in S_{\|\cdot\|}$, $h_0\in X$ and $f_0,g_0\in X^*$ with $1\leq\|h_0\|,\|g_0\|^*\leq C$ and $\|f_0\|^*\leq C$ satisfying 
\begin{align*}
    \langle f_0,h_0\rangle =0,\qquad\langle g_0,h_0\rangle =1,\qquad\langle f_0,x_0\rangle\geq 1-\delta.
\end{align*}

Consider the linear projection $P_\alpha\colon X\rightarrow \text{ker}(g_0)$ given by $P_{\alpha}x=x-\langle g_0,x\rangle h_0$, whose norm is at most $1+C^2$. Write $\|\cdot\|_{g_0}\colon\text{ker}(g_0)\rightarrow\mathbb{R}^+$ to denote the restriction of the norm $\|\cdot\|$ to the subspace $\text{ker}(g_0)$. Then $B_{\|\cdot\|_{g_0}}=B_{\|\cdot\|}\cap \text{ker}(g_0)$. 

Notice that for every $x\in X$, we have that $\langle f,P_\alpha x\rangle=\langle f,x\rangle$. Now, define the convex set:
\begin{align*}
    \widehat{B}_{\alpha,\delta}&=P_{\alpha}(B_{\|\cdot\|})\cap \{P_{\alpha}x\in \text{ker}(g_0)\colon x\in X,~|\langle f_0,P_\alpha x\rangle|\leq (1-\delta)^2\}\\
    &=P_{\alpha}(B_{\|\cdot\|})\cap \{P_{\alpha}x\in \text{ker}(g_0)\colon x\in X,~|\langle f_0,x\rangle|\leq (1-\delta)^2\}.
\end{align*}
    
Write $\vertiii{\cdot}_{\widehat{B}_{\alpha,\delta}}\colon\text{ker}(g_0)\rightarrow\mathbb{R}^+$ to denote the Minkowski functional associated to $\widehat{B}_{\alpha,\delta}$; that is, the norm in $\text{ker}(g_0)$ whose unit ball is $\widehat{B}_{\alpha,\delta}$. Using that $\langle f_0,h_0\rangle =0$, we obtain that $(1-\delta)^2P_{\alpha}(B_{\|\cdot\|})\subset \widehat{B}_{\alpha,\delta}$. This implies on the one hand that $(1-\delta)^2\vertiii{P_{\alpha}x}_{\widehat{B}_{\alpha,\delta}}\leq \|x\|$ for all $x\in X$. 

On the other hand, since $(B_{\|\cdot\|}\cap\text{ker}(g_0))\subset P_{\alpha}(B_{\|\cdot\|})$, it also means that $(1-\delta)^2B_{\|\cdot\|_{g_0}}\subset \widehat{B}_{\alpha,\delta}$. Hence, using also that $\widehat{B}_{\alpha,\delta}\subset (1+C^2)B_{\|\cdot\|_{g_0}}$, we obtain the inequality
$$ (1-\delta)^2\vertiii{y}_{\widehat{B}_{\alpha,\delta}}\leq \|y\|_{g_0}\leq (1+C^2)\vertiii{y}_{\widehat{B}_{\alpha,\delta}}$$
for all $y\in\text{ker}(g_0)$.

Now, fix $\varepsilon_{\delta,C}=\frac{1-(1-\delta)^2}{(1+C^2)^2}>0$. Then we have that 
\begin{equation}
\label{Eq1_norm_conv}
(1+\varepsilon_{\delta,C})^{1/2}(1-\delta)\vertiii{y}_{\widehat{B}_{\alpha,\delta}}\leq \left((1-\delta)^2\vertiii{y}^2_{\widehat{B}_{\alpha,\delta}}+\varepsilon_{\delta,C}\|y\|_{g_0}^2\right)^{1/2}\leq \vertiii{y}_{\widehat{B}_{\alpha,\delta}} 
\end{equation}
for every $y\in \text{ker}(g_0)$. 

Finally, we define the norm:
\begin{equation}
\label{DefinitionNormAlphaDelta}
    \vertiii{x}_{\alpha,\delta,C}= \max\left\{\|x\|,\left((1-\delta)^2\vertiii{P_{\alpha}x}^2_{\widehat{B}_{\alpha,\delta}}+\varepsilon_{\delta,C}\|P_{\alpha}x\|_{g_0}^2\right)^{1/2}\right\},\qquad\text{for all }x\in X. 
\end{equation}
Using equation \eqref{Eq1_norm_conv} and the fact that $(1-\delta)^2\vertiii{P_{\alpha}x}_{\widehat{B}_{\alpha,\delta}}\leq \|x\|$ for all $x\in X$, we observe that this norm satisfies $\|\cdot\|\leq \vertiii{\cdot}_{\alpha,\delta,C}\leq \frac{1}{(1-\delta)^2}\|\cdot\|$. 

Geometrically, the unit ball of $\vertiii{\cdot}_{\alpha,\delta,C}$ is the intersection of the original ball $B_{\|\cdot\|}$ with the cylinder generated by the unit ball of the norm $((1-\delta)^2\vertiii{\cdot}^2_{\widehat{B}_{\alpha,\delta}}+\varepsilon_{\delta,C}\|\cdot\|^2)^{1/2}$ (defined in $\text{ker}(g_0)$) in the direction of $h_0$.

Let us collect some properties of the norm $\vertiii{\cdot}_{\alpha,\delta,C}$ which are easily shown.

\begin{lemma}
\label{Lemma1.Properties_norm_1}
Let $X$ be a Banach space, and let $\delta$, $C$, and $\alpha$ be as above. The norm $\vertiii{\cdot}_{\alpha,\delta,C}$ satisfies:

\begin{itemize}
    \item[(a)] If $\lambda_0 = \frac{1}{\vertiii{x_0}_{\alpha,\delta,C}}$, we have that $(1-\delta)^2\leq \lambda_0\leq \frac{1}{(1+\varepsilon_{\delta,C})^{1/2}}$, where $\varepsilon_{\delta,C}=\frac{1-(1-\delta)^2}{(1+C^2)^2}>0$.

    \item[(b)] The unit sphere $S_{\vertiii{\cdot}_{\alpha,\delta,C}}$ contains the segment $\left[\lambda_0 x_0-\frac{(1-\lambda_0)}{C}h_0,\lambda_0 x_0+\frac{(1-\lambda_0)}{C}h_0\right]$ where $\lambda_0=\frac{1}{\vertiii{x_0}_{\alpha,\delta,C}}$.
    \item[(c)] $\vertiii{x}_{\alpha,\delta,C}=\|x\|$ for all $x\in X$ with $|\langle f_0,x\rangle|\leq (1-\delta)^2\|x\|$.
\end{itemize}
\end{lemma}
\begin{proof}

    To obtain part $(a)$, first notice that $\vertiii{x_0}_{\alpha,\delta,C}\leq \frac{1}{(1-\delta)^2}$ and thus $\lambda_0\geq (1-\delta)^2$. On the other hand, since $\langle f_0,P_{\alpha}x_0\rangle \geq 1-\delta$, we have by definition of $\vertiii{\cdot}_{\widehat{B}_{\alpha,\delta}}$ that $\vertiii{(1-\delta)P_\alpha x_0}_{\widehat{B}_{\alpha,\delta}}\geq 1$. Then, we can apply equation \eqref{Eq1_norm_conv} to obtain that $\lambda_0\leq \frac{1}{(1+\varepsilon_{\delta,C})^{1/2}}$.

    For item $(b)$, fix $\mu \in\left[-\frac{(1-\lambda_0)}{C},\frac{(1-\lambda_0)}{C}\right]$. Since $\langle f_0,h_0\rangle =0$, we have that $P_{\alpha}(\lambda_0x_0)=P_\alpha(\lambda_0x_0 +\mu h_0)$. Moreover, the point $\lambda_0 x_0+\mu h_0$ belongs to the unit ball $B_{\|\cdot\|}$. It follows now from the definition of $\vertiii{\cdot}_{\alpha,\delta,C}$ that $\vertiii{\lambda_0 x_0+\mu h_0}_{\alpha,\delta,C}=\vertiii{\lambda_0x_0}_{\alpha,\delta,C}=1$.  

    We only need to check the equality in part $(c)$ for points in the unit sphere $S_{\|\cdot\|}$. For every $x\in X$ with $\|x\|=1$ we have that $\vertiii{x}_{\alpha,\delta,C}\geq 1$. If moreover $|\langle f_0,x\rangle|\leq (1-\delta)^2$, the point $P_{\alpha}x$ belongs to $\widehat{B}_{\alpha,\delta}$, and thus by equation \eqref{Eq1_norm_conv} it is clear that $\vertiii{x}_{\alpha,\delta,C}=1$.
\end{proof}

\subsection*{The norm $\vertiii{\cdot}_{\Omega}$}\hfill\\

In this subsection we use countably many of the previous norms to define the renormings to prove the main theorems of the section. We will define it abstractly using a tuple $\Omega=\left(\delta,C,\{\alpha_n\}_{n\in\mathbb{N}},\{\eta_n\}_{n\in\mathbb{N}}\right)$, where $0<\delta<1$ with $(1-\delta)^4>\frac{1}{2}$, the constant $C$ is bigger than $1$, the set $\alpha_n=(x_n,h_n,f_n,g_n)$ is a $4$-tuple with $x_n\in S_{\|\cdot\|}$, $h_n\in X$ and $f_n,g_n\in X^*$ with $1\leq\|h_n\|,\|g_n\|^*\leq C$ and $\|f_n\|^*\leq C$, such that 
\begin{align*}
    \langle f_n,h_n\rangle =0,\qquad\langle g_n,h_n\rangle =1,\qquad\langle f_n,x_n\rangle\geq 1-\delta
\end{align*}
for every $n\in\mathbb{N}$, and $\{\eta_n\}_{n\in\mathbb{N}}$ is a decreasing sequence of strictly positive numbers converging to $0$. Moreover, we also suppose that $|\langle f_m,x_n\rangle| < 2(1-\delta)^4-1$ for every $m\neq n\in\mathbb{N}$, and that for every $x\in X\setminus\{0\}$ there exists an open neighbourhood $U_x$ of $x$ and $n_x\in\mathbb{N}$ such that $|\langle f_n,z\rangle|\leq(1-\delta)^2\|z\|$ for all $z\in U_x$ and $n\geq n_x$.

Fix $n\in\mathbb{N}$, and define $\vertiii{\cdot}_{\alpha_n,\delta,C}$ to be the equivalent norm in $X$ we discussed in the previous subsection. 

Put $\tau_n=\frac{(1+\eta_n)^2-1}{(1+\eta_n)^2}$, and set 
\begin{equation}
    \label{NormN}
    \vertiii{x}_n=\left(\frac{1}{(1+\eta_n)^2}\vertiii{x}_{\alpha_n,\delta,C}^2+\tau_n\|x\|^2\right)^{1/2},\qquad\text{for all }x\in X.
\end{equation}    
This is an equivalent norm in $X$, satisfying
\begin{equation}
    \label{NormN_ineq}\frac{1}{1+\eta_n}\vertiii{\cdot}_{\alpha_n,\delta,C}\leq\vertiii{\cdot}_n\leq \vertiii{\cdot}_{\alpha_n,\delta,C}.
\end{equation}
Moreover, by Lemma \ref{Lemma5:Averageofrotundisrotund}, if the norm $\|\cdot\|$ is UR (resp. WUR, URED, LUR), then $\vertiii{\cdot}_n$ is also UR (resp. WUR, URED, LUR). 

The final renorming we will need is defined as follows:
\begin{equation}
    \label{NormOmega}
    \vertiii{x}_\Omega = \max\{\|x\|,\max_{n\in\mathbb{N}}\vertiii{x}_n\}\qquad\text{for all }x\in X.
\end{equation}
Clearly $\vertiii{0}_\Omega=0$. Recall that given a point $x\in X\setminus\{0\}$, we have that $|\langle f_n,z\rangle|\leq(1-\delta)^2\|z\|$ for all $z\in U_x$ and all $n\geq n_x$. Then, item $(c)$ in Lemma \ref{Lemma1.Properties_norm_1} applied to each $\vertiii{\cdot}_{\alpha_n,\delta,C}$ for $n\geq n_x$ implies that $\vertiii{z}_{n}\leq \|z\|$ for all $z$ in the open neighbourhood $U_x$ of $x$. This means that $\vertiii{\cdot}_\Omega$ coincides with the finite intersection of norms given by $ \max\{\|\cdot\|,\max_{i=1,\dots,n_x}\vertiii{\cdot}_i\}$ in an open neighbourhood of $x$. Therefore, $\vertiii{\cdot}_\Omega$ is well defined, and if $\|\cdot\|$ is LUR, then each $\vertiii{\cdot}_n$ and thus $\vertiii{\cdot}_\Omega$ is LUR as well by Lemma \ref{Lemma6:Maximumofrotund}. It is also straightforward to obtain that $\|\cdot\|\leq\vertiii{\cdot}_\Omega\leq\frac{1}{(1-\delta)^2}\|\cdot\|$. 

Although the specific properties of the norm $\vertiii{\cdot}_\Omega$ depend on the choice of the sequence $\{\alpha_n\}_{n\in\mathbb{N}}$ and the qualities of the original norm in the space $X$, we can prove in general a statement showing that $\vertiii{\cdot}_\Omega$ fails to be uniformly rotund using segments in the directions of the sequence $\{h_n\}_{n\in\mathbb{N}}$:

\begin{lemma}
    \label{Lemma2.Norm_A_notUR}
    Let $(X,\|\cdot\|)$ be a Banach space, and let $\Omega$ be as above. Put $\lambda_n= \frac{1}{\vertiii{x_n}_{\alpha_n,\delta,C}}$ for every $n\in\mathbb{N}$. Then, for every sequence $\{z_n\}_{n\in\mathbb{N}}$ such that 
    $$z_n\in\left[\lambda_n x_n-\frac{(1-\lambda_n)}{C}h_n,\lambda_n x_n+\frac{(1-\lambda_n)}{C}h_n\right]$$ 
    for all $n\in\mathbb{N}$, we have that $\vertiii{z_n}_\Omega\rightarrow 1$.
\end{lemma}
\begin{proof}
    Fix $n\in\mathbb{N}$. Using $(b)$ in Lemma \ref{Lemma1.Properties_norm_1} and the definition of $\vertiii{\cdot}_n$ above, we have that $ \frac{1}{1+\eta_n}\leq \vertiii{z_n}_n\leq 1$. Hence, we obtain that $\vertiii{z_n}_\Omega\geq \frac{1}{1+\eta_n}$. Since $\eta_n\rightarrow 0$ and $\|z_n\|\leq 1$, the result will follow if we show that $\vertiii{z_n}_m\leq 1$ for all $m\neq n$. 

    Fix $m\neq n$. Observe that $\lambda_n\geq (1-\delta)^2$ by $(a)$ in Lemma \ref{Lemma1.Properties_norm_1}, and thus $\|z_n\|\geq 2(1-\delta)^2-1$. Then, using that $|\langle f_m,x_n\rangle|<2(1-\delta)^4-1$ and $\lambda_n\leq 1$ we obtain:
    \begin{align*}
        |\langle f_m,z_n\rangle|&\leq |\langle f_m,\lambda_nx_n\rangle|+|\langle f_m,\lambda_nx_n-z_n\rangle|\\
        &\leq 2(1-\delta)^4-(1-\delta)^2\leq(1-\delta)^2\|z_n\|. 
    \end{align*}
    Therefore, by $(c)$ in Lemma \ref{Lemma1.Properties_norm_1}, it holds that $\vertiii{z_n}_{\alpha_m,\delta,C}=\|z_n\|\leq 1$, from which it follows that $\vertiii{z_n}_m\leq 1$, and the proof is finished.
\end{proof}

\subsection*{Proofs of the Main Theorems of Section 3}
Each of the three main theorems of the section will be proven using a different tuple $\Omega$ to define the desired renorming. To aid in the definition of these elements, we use a general result of Banach spaces. The following lemma is a standard result with a short proof that we include for completeness. This is also used and proven in \cite{Qui22} (Claim 1).
\begin{lemma}
\label{Lemma3.Almostbiorthogonal}
Let $(X,\|\cdot\|)$ be an infinite dimensional Banach space, and let $\{\varepsilon_n\}_{n\in\mathbb{N}}$ be a decreasing sequence of positive numbers. Then there exist a sequence $\{y_n\}_{n\in\mathbb{N}}$ in $S_{\|\cdot\|}$ and a weak$^*$ null sequence $\{\varphi_n\}_{n\in\mathbb{N}}$ in $S_{\|\cdot\|^*}$ such that $\langle \varphi_n,y_n\rangle =1$ for all $n\in\mathbb{N}$, and $|\langle \varphi_m,y_n\rangle|<\varepsilon_{\min\{m,n\}}$ for all $m\neq n\in\mathbb{N}$.
\end{lemma}
\begin{proof}
    Let $\{\varphi_n\}_{n\in\mathbb{N}}$ be a weak$^*$ null sequence in $S_{\|\cdot\|^*}$, which exists by Josefson$-$Nissenzweig's Theorem. We will inductively define a subsequence $\{\varphi_{n_k}\}_{k\in\mathbb{N}}$ and a sequence $\{y_k\}_{k\in\mathbb{N}}$ in $S_{\|\cdot\|}$ such that
    \begin{itemize}
        \item[(i)]$\langle \varphi_{n_k},y_k\rangle\geq 1-2^{-k}$,
        \item[(ii)]$|\langle \varphi_{n_j},y_i\rangle| <\varepsilon_{n_j}/2$ for all $i,j\leq k$ with $i<j$,
        \item[(iii)]$\langle \varphi_{n_i},y_j\rangle=0$ for all $i,j\leq k$ with $i<j$.   
    \end{itemize}
    Put $n_1=1$ and choose $y_1\in S_{\|\cdot\|}$ with $\langle \varphi_1,y_1\rangle\geq 1/2$, and suppose we have defined the first $k$ terms of the two sequences. 
    
    For each $i\leq k$, define the linear projection $P_i\colon X\rightarrow \text{span}\{y_i\}$ given by $P_ix=\frac{\langle \varphi_{n_i},x\rangle}{\langle \varphi_{n_i,y_i}\rangle}y_i$. Setting $T_1=P_1$ we may define inductively the linear projection $T_{i+1}\colon X\rightarrow \text{span}\{y_1,\dots,y_{i+1}\} $ by $T_{i+1}=T_i+P_{i+1}\circ(I-T_i)$ for all $1\leq i\leq k-1$ (we need to use that $(iii)$ holds to show that this map is a linear projection). It can be shown as well that $\langle \varphi_{n_i},(I-T_k)x\rangle=0$ for all $i\leq k$ and all $x\in X$. 

    We claim that there exists $n_{k+1}\in\mathbb{N}$ which can be taken to be arbitrarily large such that
    $$ \text{sup}_{x\in S_{\|\cdot\|}}\langle \varphi_{n_{k+1}},(I-T_k)x\rangle \geq 1-2^{-(k+1)}.$$
    Indeed, otherwise, choosing a sequence $\{x_n\}_{n\in\mathbb{N}}$ in $B_{\|\cdot\|}$ such that $\langle \varphi_n,x_n\rangle\geq 1-2^{-(k+2)}$ we would obtain another bounded sequence $\{T_kx_n\}_{n\in\mathbb{N}}$ in the finite dimensional space $\text{span}\{y_1,\dots,y_k\}$ such that $\langle \varphi_n,T_kx_n\rangle \geq 2^{-(k+2)}$ holds for infinitely many $n\in\mathbb{N}$. Using the relative compactness of such a subsequence we arrive at a contradiction with the fact that $\{\varphi_n\}_{n\in\mathbb{N}}$ is weak$^*$ null. 

    We may take $n_{k+1}$ large enough such that $|\langle \varphi_{n_{k+1}},y_i\rangle|<\varepsilon_{n_
    {k+1}}/2$ for all $i\leq k$. Now, choose $y_{k+1}\in S_{\|\cdot\|}$ satisfying $\langle \varphi_{n_{k+1}},y_{k+1}\rangle \geq 1-2^{-(k+1)}$ and $\langle \varphi_{n_i},y_{k+1}\rangle =0$ for all $i\leq k$. This finishes the induction. 

    We can now apply Bishop$-$Phelps$-$Bollobás' Theorem (as stated in e.g.: Page 376 in \cite{FabHabHajMonZiz11}) to the sequences $\{\varphi_{n_k}\}_{k\in\mathbb{N}}$ and $\{y_k\}_{k\in\mathbb{N}}$ to obtain the desired result.
\end{proof}
We proceed to the proof of the main theorems of the section:

The proof of Theorem \ref{Main_Theorem_A} is the most straightforward of the three. We will construct a tuple $\Omega$ of the above form such that the sequence $\{h_n\}_{n\in\mathbb{N}}$ is actually constantly equal to a fixed vector $h_0$, which will be the direction in which the norm $\vertiii{\cdot}_\Omega$ fails to be Uniformly Rotund, and thus failing to be URED. As we have seen, the choice of the tuple $\Omega$ and the definition of $\vertiii{\cdot}_\Omega$ ensures that local properties are kept, and thus LUR is preserved.

\begin{proof}[Proof of Theorem \ref{Main_Theorem_A}]
    Let $X$ be an infinite dimensional Banach space and let $\|\cdot\|$ be a LUR norm on $X$. Let $0<\delta<1$ such that $(1-\delta)^4>\frac{1}{2}$.
    
    Using Lemma \ref{Lemma3.Almostbiorthogonal}, we can find a sequence $\{y_n\}_{n\in\mathbb{N}}$ in $S_{\|\cdot\|}$ and a weak$^*$ null sequence $\{\varphi_n\}_{n\in\mathbb{N}}$ in $S_{\|\cdot\|^*}$ such that $\langle \varphi_n,y_n\rangle =1$ for all $n\in\mathbb{N}$, and $|\langle \varphi_m,y_n\rangle|<\min\left\{\frac{2(1-\delta)^4-1}{2},\delta\right\}$ for all $m\neq n\in\mathbb{N}$.

    Putting $h_n=y_1$ and $g_n=\varphi_1$ for all $n\in\mathbb{N}$, and setting $x_n=y_{n+1}$ and $f_n=\varphi_{n+1}-\langle\varphi_{n+1},y_1\rangle\varphi_1$ for all $n\in\mathbb{N}$, we have that $x_n,h_n\in S_{\|\cdot\|}$ and $f_n,g_n\in S_{\|\cdot\|^*}$. Moreover, we obtain that
    \begin{align*}
        \langle f_n,h_n\rangle =0,\qquad\langle g_n,h_n\rangle =1,\qquad\langle f_n,x_n\rangle\geq 1-\delta
    \end{align*}
    with $|\langle f_m,x_n\rangle| < 2(1-\delta)^4-1$ for every $m\neq n\in\mathbb{N}$. Since the sequence $\{f_n\}_{n\in\mathbb{N}}$ is weak$^*$ null, we have that for every $x\in X\setminus\{0\}$ there exists an open neighbourhood $U_x$ of $x$ and $n_x\in\mathbb{N}$ such that $|\langle f_n,z\rangle|\leq(1-\delta)^2\|z\|$ for all $z\in U_x$ and all $n\geq n_x$. 

    Therefore, setting $\alpha_n=(x_n,h_n,f_n,g_n)$ for every $n\in\mathbb{N}$ and $\Omega = (\delta, C,\{\alpha_n\}_{n\in\mathbb{N}}, \{2^{-n}\}_{n\in\mathbb{N}})$ with $C=1+\delta$ we are able to define the norms $\vertiii{\cdot}_{\alpha_n,\delta,C}$ and $\vertiii{\cdot}_n$ for every $n\in\mathbb{N}$, and the norm $\vertiii{\cdot}_\Omega$ in $(X,\|\cdot\|)$ as in equations \eqref{DefinitionNormAlphaDelta},\eqref{NormN} and \eqref{NormOmega}. Then we have that 
    $$\|\cdot\|\leq\vertiii{\cdot}_\Omega\leq\frac{1}{(1-\delta)^2}\|\cdot\|. $$
    Since $\delta$ can be taken to be arbitrarily small and LUR norms are dense in $X$, we obtain that every norm in $X$ can be approximated uniformly on bounded sets by norms of this form.

    As we mentioned after the definition of $\vertiii{\cdot}_\Omega$, if $\|\cdot\|$ is LUR, then so is $\vertiii{\cdot}_\Omega$, and thus it only remains to show that this last norm is not URED to finish the proof. Indeed, using Lemma \ref{Lemma2.Norm_A_notUR} we have that setting $\lambda_n=\frac{1}{\vertiii{x_n}_{\alpha_n,\delta,C}}$ and 
    \begin{align*}
        a_n&= \lambda_n x_n-\frac{(1-\lambda_n)}{1+\delta}y_1\\
        b_n&= \lambda_n x_n+\frac{(1-\lambda_n)}{1+\delta}y_1,
    \end{align*} 
    the sequences $\{\|a_n\|_\Omega\}_{n\in\mathbb{N}}$, $\{\|b_n\|_\Omega\}_{n\in\mathbb{N}}$ and $\left\{\vertiii{\frac{a_n+b_n}{2}}_\Omega\right\}_{n\in\mathbb{N}}$ converge to $1$, while $b_n-a_n=2\frac{(1-\lambda_n)}{1+\delta}y_1$. By $(i)$ in Lemma \ref{Lemma1.Properties_norm_1}, we have that $\lambda_n\leq \frac{1}{(1+\varepsilon_{\delta,1})^{1/2}}<1$ for all $n\in\mathbb{N}$, and thus we conclude that $\vertiii{\cdot}_\Omega$ is not Uniformly Rotund in the direction of $y_1$, and hence fails to be URED.

\end{proof}

For the proof of Theorem \ref{Main_Theorem_B}, we will need a sequence of directions $\{h_n\}_{n\in\mathbb{N}}$ which is not weakly null but which forms a uniformly separated set. In this way, using again Lemma \ref{Lemma3.Almostbiorthogonal}, the resulting norm $\vertiii{\cdot}_\Omega$ will not be WUR, though it will still be URED as we are going to show. 

\begin{proof}[Proof of Theorem \ref{Main_Theorem_B}]
    Let $X$ be an infinite dimensional Banach space, and suppose $X$ admits a URED and LUR norm $\|\cdot\|$. Let $0<\delta<1$ such that $(1-\delta)^4>\frac{1}{2}$.
    
    Fix $v_0\in S_{\|\cdot\|}$ and $\psi_0\in S_{\|\cdot\|^*}$ such that $\langle\psi_0,v_0\rangle=1$. We define the projection $P_0\colon X\rightarrow \text{ker}(\psi_0)$ given by $P_0x=x-\langle\psi_0,x\rangle v_0$. If we consider the infinite dimensional subspace $(\text{ker}(\psi_0),\|\cdot\|)$ with the norm inherited from $(X,\|\cdot\|)$, the projection $P_0$ has norm at most $2$. 
    
    We can apply Lemma \ref{Lemma3.Almostbiorthogonal} to obtain a sequence $\{y_n\}_{n\in\mathbb{N}}$ in $S_{(\text{ker}(\psi_0),\|\cdot\|)}$ and a weak$^*$ null sequence $\{\varphi_n\}_{n\in\mathbb{N}}$ in $S_{(\text{ker}(\psi_0)^*,\|\cdot\|^*)}$, such that $\langle\varphi_n,y_n\rangle=1$ for all $n\in\mathbb{N}$ and $|\langle \varphi_m,y_n\rangle|<\min\left\{\frac{2(1-\delta)^4-1}{2},\delta\right\}$ for all $m\neq n\in\mathbb{N}$.

    Now, we define for every $n\in\mathbb{N}$ the following vectors and functionals:
    \begin{align*}
        x_n&= y_{2n}&&\in S_{\|\cdot\|},\\ 
        h_n&= y_{2n+1}+v_0&&\in X,\\
        f_n&= \left(\varphi_{2n}-\langle \varphi_{2n},y_{2n+1}\rangle\varphi_{2n+1}\right)\circ P_0&&\in X^*,\\
        g_n&= \varphi_{2n+1}\circ P_0&&\in X^*,
    \end{align*}
    which satisfy that $1\leq \|h_n\|,\|g_n\|^*\leq 2$ and $\|f_n\|^*\leq 2(1+\delta)$. We also have that
    $$ \langle f_n,h_n\rangle =0,\qquad\langle g_n,h_n\rangle=1,\qquad\langle f_n,x_n\rangle\geq 1-\delta$$
    for every $n\in\mathbb{N}$, while for every $m\neq n\in\mathbb{N}$ it holds that $\langle f_m,x_n\rangle< 2(1-\delta)^4-1$. Importantly, the sequences $\{f_n\}_{n\in\mathbb{N}}$ and $\{g_n\}_{n\in\mathbb{N}}$ are still weak$^*$ null in $X^*$, and thus in particular for every $x\in X\setminus\{0\}$ there exists an open neighbourhood $U_x$ of $x$ and $n_x\in\mathbb{N}$ such that $|\langle f_n,z\rangle|\leq(1-\delta)^2\|z\|$ for all $z\in U_x$ and all $n\geq n_x$. Finally, notice as well that the sequence $\{h_n\}_{n\in\mathbb{N}}$ is not weakly null, since $\langle h_n,\psi_0\rangle =1$ for all $n\in\mathbb{N}$. 

    We set now $\Omega=(\delta,C,\{\alpha_n\}_{n\in\mathbb{N}},\{2^{-n}\}_{n\in\mathbb{N}})$, where $C=2(1+\delta)$ and $\alpha_n=(x_n,h_n,f_n,g_n)$ for all $n\in\mathbb{N}$. We can define the norms $\vertiii{\cdot}_{\alpha_n,\delta,C}$ and $\vertiii{\cdot}_{n}$ for every $n\in\mathbb{N}$ as in equations \eqref{DefinitionNormAlphaDelta},\eqref{NormN}, as well as the final renorming $\vertiii{\cdot}_\Omega$ as in equation \eqref{NormOmega}. The density of norms of this form follows as in the previous theorem. 
    
    As discussed in the definition of these norms, we have that $\vertiii{\cdot}_n$ is URED and LUR for every $n\in\mathbb{N}$, and $\vertiii{\cdot}_\Omega$ is LUR. 
    
    Similarly to the proof of Theorem \ref{Main_Theorem_A}, we define $\lambda_n=\frac{1}{\vertiii{x}_{\alpha_n,\delta,C}}$ and the pair $a_n=\lambda_nx_n+\frac{1-\lambda_n}{2(1+\delta)}h_n$ and $b_n=\lambda_nx_n-\frac{1-\lambda_n}{2(1+\delta)}h_n$, which show in combination with Lemma \ref{Lemma2.Norm_A_notUR} and $(i)$ in Lemma \ref{Lemma1.Properties_norm_1} that $\vertiii{\cdot}_\Omega$ is not WUR, since $\{\vertiii{a_n}_\Omega\}_{n\in\mathbb{N}}$, $\{\vertiii{b_n}_\Omega\}_{n\in\mathbb{N}}$ and $\left\{\vertiii{\frac{a_n+b_n}{2}}_\Omega\right\}_{n\in\mathbb{N}}$ all converge to $1$ but $\{b_n-a_n\}_{n\in\mathbb{N}}$ is not a weakly null sequence. 
    
    It only remains to show that $\vertiii{\cdot}_\Omega$ is URED. Fix $h\in S_{\|\cdot\|}$ and consider two sequences $\{c_k\}_{n\in\mathbb{N}}$ and $\{d_k\}_{n\in\mathbb{N}}$ in $S_{\vertiii{\cdot}_\Omega}$ with $d_k-c_k=\xi_k h$ for some $\xi_k\in \mathbb{R}$ for all $k\in\mathbb{N}$ and such that $\vertiii{\frac{c_k+d_k}{2}}_{\Omega}\rightarrow 1$. Notice that the sequence $\{\xi_k\}_{k\in\mathbb{N}}$ is bounded. Suppose by contradiction that $\{\xi_k\}_{k\in\mathbb{N}}$ does not converge to $0$. Then, by passing to a subsequence, we may assume that there exists $\varepsilon_0>0$ such that $|\xi_k|\geq \varepsilon_0$ for all $k\in\mathbb{N}$. 
    
    Define $\vertiii{\cdot}_{\Omega_n}=\max_{m\leq n}\{\|\cdot\|,\vertiii{\cdot}_m\}$ for all $n\in\mathbb{N}$, which is an equivalent URED norm in $X$ by Lemma \ref{Lemma6:Maximumofrotund}. Suppose first that there exists $n_0\in\mathbb{N}$ such that $\vertiii{\frac{c_k+d_k}{2}}_\Omega=\vertiii{\frac{c_k+d_k}{2}}_{\Omega_{n_0}}$ for all $k\in\mathbb{N}$. Then, since $\vertiii{c_k}_{\Omega_{n_0}},\vertiii{d_k}_{\Omega_{n_0}}\leq 1$ for all $k\in\mathbb{N}$ and $\vertiii{\cdot}_{\Omega_{n_0}}$ is URED, this already leads to the desired contradiction. 

    Hence, by passing to a subsequence again, we may suppose that $\vertiii{\frac{c_k+d_k}{2}}_\Omega>\left\|\frac{c_k+d_k}{2}\right\|$ for every $k\in\mathbb{N}$, and that there exists a sequence $\{n_k\}_{k\in\mathbb{N}}$ of natural numbers with $n_k\rightarrow \infty$ and such that $\vertiii{\frac{c_k+d_k}{2}}_\Omega=\vertiii{\frac{c_k+d_k}{2}}_{n_k}$ for all $k\in\mathbb{N}$. Equation \eqref{NormN_ineq} shows that for all $k\in\mathbb{N}$ we have
    $$ \vertiii{\frac{c_k+d_k}{2}}_{n_k}\leq\vertiii{\frac{c_k+d_k}{2}}_{\alpha_{n_k},\delta,C}\leq \frac{1}{1+2^{-n_k}}\vertiii{\frac{c_k+d_k}{2}}_{n_k}. $$
    Now, since $2^{-n_k}\rightarrow 0$, with the definition of $\vertiii{\cdot}_{\alpha_{n_k},\delta,C}$ (see equation \eqref{DefinitionNormAlphaDelta}) and noting that $\vertiii{\frac{c_k+d_k}{2}}_{\alpha_{n_k},\delta,C}>\left\|\frac{c_k+d_k}{2}\right\|$ for all $k\in\mathbb{N}$, we deduce that 
    \begin{equation}
    \label{Lim_middle_points}
    \left((1-\delta)^2\vertiii{P_{\alpha_{n_k}}\left(\frac{c_k+d_k}{2}\right)}^2_{\widehat{B}_{\alpha_{n_k},\delta}}+\varepsilon_{\delta,C}\left\|P_{\alpha_{n_k}}\left(\frac{c_k+d_k}{2}\right)\right\|^2\right)^{1/2}\longrightarrow 1, 
    \end{equation}
    where $P_{\alpha_{n_k}}$ is the projection from $X$ onto $\text{ker}(g_{n_k})$ given by $P_{\alpha_{n_k}}x=x-\langle g_{n_k},x\rangle h_{n_k}$, and $\varepsilon_{\delta,C}=\frac{1-(1-\delta)^2}{(1+C^2)^2}>0$. Note that, importantly, the constant $\varepsilon_{\delta,C}$ does not depend on $k\in\mathbb{N}$.

    For every $k\in\mathbb{N}$, the point $c_k$ is in $S_{\vertiii{\cdot}_\Omega}$, so we obtain that $\vertiii{c_k}_{\alpha_{n_k},\delta,C}\leq \frac{1}{1+2^{-n_k}}$, which implies that $\limsup_{k\rightarrow\infty}\vertiii{c_k}_{\alpha_{n_k},\delta,C}\leq 1$. Similarly we obtain that $\limsup_{k\rightarrow\infty}\vertiii{d_k}_{\alpha_{n_k},\delta,C}\leq 1$. Therefore, we have that 
  \begin{align*}
      \limsup_{k\rightarrow\infty}\left((1-\delta)^2\vertiii{P_{\alpha_{n_k}}c_k}^2_{\widehat{B}_{\alpha_{n_k},\delta}}+\varepsilon_{\delta,C}\left\|P_{\alpha_{n_k}}c_k\right\|^2\right)^{1/2}&\leq 1,\text{ and }\\
      \limsup_{k\rightarrow\infty}\left((1-\delta)^2\vertiii{P_{\alpha_{n_k}}d_k}^2_{\widehat{B}_{\alpha_{n_k},\delta}}+\varepsilon_{\delta,C}\left\|P_{\alpha_{n_k}}d_k\right\|^2\right)^{1/2}&\leq 1
  \end{align*}

    Hence, using also equation \eqref{Lim_middle_points}, we can apply Remark \ref{remark_Q} and equation \eqref{LinearityQ}, together with the non-negativity of the function $Q_{\|\cdot\|}$ to obtain that $Q_{\|\cdot\|}(P_{\alpha_{n_k}}c_k,P_{\alpha_{n_k}}d_k)\rightarrow 0$. Define $v_k=\xi_k\langle g_{n_k},h\rangle h_{n_k}\in X$ for every $k\in\mathbb{N}$. Then $\{v_k\}_{k\in\mathbb{N}}$ converges to $0$ in norm because $\{\xi_k\}_{k\in\mathbb{N}}$ is bounded and $\{g_{n_k}\}_{k\in\mathbb{N}}$ is weak$^*$ null. Therefore $Q_{\|\cdot\|}(P_{\alpha_{n_k}}c_k,P_{\alpha_{n_k}}d_k+v_k)\rightarrow 0$ and moreover $P_{\alpha_{n_k}}d_k+v_k-P_{\alpha_{n_k}}c_k=\xi_k h$ for all $k\in\mathbb{N}$. Since $\|\cdot\|$ is URED, we apply Proposition \ref{Prop2.URED_Q_Characterization} and conclude that $\{\xi_k\}_{k\in\mathbb{N}}$ converges to $0$. This leads to a contradiction and the proof is finished.

\end{proof}

Finally, in the superreflexive case, the set $\{h_n\}_{n\in\mathbb{N}}$ of directions in which we approximate the segments in the sphere will consist of a weakly null sequence. We will then obtain WUR without UR similarly as in the previous proof. Note however that, as we will see, the additional requirement that $\{h_n\}_{n\in\mathbb{N}}$ is weakly null impedes us to easily ensure that the sequence of orthogonal functionals $\{f_n\}_{n\in\mathbb{N}}$ is weak$^*$ null as well. This property was crucial in both previous proofs in order to show that $\vertiii{\cdot}_\Omega$ is well defined. To solve this, we will employ a simple geometric argument which is possible due to the existence of a uniformly rotund norm.

In a Banach space $X$, given a closed and convex set $C\subset X$, a functional $f\in X^*$, and $r>0$, we define the set $S(C,f,\varepsilon)=\{x\in C\colon \langle f,x\rangle>r\}$. Sets of this form are called \emph{slices of $C$}. Recall that a norm $\|\cdot\|$ in a Banach space is uniformly rotund if and only if for every $\varepsilon>0$ there exists $\delta>0$ such that for every $f\in S_{\|\cdot\|^*}$, the slice $S(B_{\|\cdot\|},f,1-\delta)$ has diameter less than $\varepsilon$.

\begin{proof}[Proof of Theorem \ref{Main_Theorem_C}]
    Let $X$ be an infinite dimensional superreflexive Banach space with a UR norm $\|\cdot\|$, and let $0<\delta<1$ such that $(1-\delta)^4>\frac{1}{2}$. Additionally, using that $\|\cdot\|$ is UR, we may take $\delta$ small enough such that the diameter of the slice $S(B_{\|\cdot\|},f,(1-\delta)^3)$ is less than $\frac{1}{4}$ for all $f\in X^*$ with $\|f\|^*\leq 2$.
    
    Using reflexivity, we may apply Lemma \ref{Lemma3.Almostbiorthogonal} to $(X^*,\|\cdot\|^*)$ to obtain a sequence $\{y_n\}_{n\in\mathbb{N}}$ in $S_{\|\cdot\|^*}$ and a weakly null sequence $\{\varphi_n\}_{n\in\mathbb{N}}$ in $S_{\|\cdot\|}$ such that $\langle y_n,\varphi_n\rangle =1$ for all $n\in\mathbb{N}$ and $\langle y_m,\varphi_n\rangle <\min\left\{\frac{2(1-\delta)^4-1}{2},\delta\right\}$ for all $m\neq n\in\mathbb{N}$. Fix $n\in\mathbb{N}$ and define
    \begin{align*}
        x_n&= \varphi_{2n}&&\in S_{\|\cdot\|},\\ 
        h_n&= \varphi_{2n+1}&&\in S_{\|\cdot\|},\\
        f_n&= y_{2n}-\langle y_{2n},\varphi_{2n+1}\rangle y_{2n+1}&&\in X^*,\\
        g_n&= y_{2n+1}&&\in S_{\|\cdot\|^*}.
    \end{align*}
    It holds that $\|f_n\|^*\leq 1+\delta$ and
    $$ \langle f_n,h_n\rangle =0,\qquad\langle g_n,h_n\rangle=1,\qquad\langle f_n,x_n\rangle\geq 1-\delta.$$
    Notice as well that the sequence $\{h_n\}_{n\in\mathbb{N}}$ is weakly null. Additionally, for $m\neq n\in \mathbb{N}$ we have that $\langle f_m,x_n\rangle <2(1-\delta)^4-1$. 
    
    Set $\alpha_n=(x_n,h_n,f_n,g_n)$ for every $n\in\mathbb{N}$ and $\Omega=(\delta,C,\{\alpha_n\}_{n\in\mathbb{N}},\{2^{-n}\}_{n\in\mathbb{N}})$ with $C=1+\delta$. In order to define the norm $\vertiii{\cdot}_\Omega$ properly, it only remains to show that for every $x\in X\setminus\{0\}$ there exists an open neighbourhood $U_x$ of $x$ and $n_x\in\mathbb{N}$ such that $|\langle f_n,z\rangle|\leq(1-\delta)^2\|z\|$ for all $z\in U_x$ and all $n\geq n_x$.  We will show that for any $x\in S_{\|\cdot\|}$ there exists at most one $n\in\mathbb{N}$ such that $|\langle f_n,x\rangle| >(1-\delta)^3$, which by homogeneity and by the fact that $(1-\delta)^3<(1-\delta)^2$, implies the condition we need.

    Indeed, suppose there existed such $x\in S_{\|\cdot\|}$ and $m\neq n\in\mathbb{N}$ such that $|\langle f_m,x\rangle|>(1-\delta)^3$ and $|\langle f_n,x\rangle|>(1-\delta)^3$. We assume that $\langle f_j,x\rangle >0$ for $j\in\{m,n\}$, since the other possibilities are proven similarly. Then $x$ belongs to both slices $S(B_{\|\cdot\|},f_m,(1-\delta)^3)$ and $S(B_{\|\cdot\|},f_n,(1-\delta)^3)$, which have diameter $1/4$ by choice of $\delta$. Since the point $x_j$ also belongs to the slice $S(B_{\|\cdot\|},f_j,(1-\delta)^3)$ for both $j\in\{m,n\}$, we obtain using the triangle inequality that $\|x_m-x_n\|\leq \frac{1}{2}$. However, considering the functional $y_{2n}\in S_{\|\cdot\|^*}$ we get:
    $$ \|x_n-x_m\|\geq |\langle y_{2n},x_n\rangle|-|\langle y_{2n},x_m\rangle|>1-\delta>\frac{1}{2},$$
    a contradiction. 

    Hence, we can define the norm $\vertiii{\cdot}_\Omega$ as in equation \eqref{NormOmega}, which is LUR. We also consider the norms $\vertiii{\cdot}_{\alpha_n,\delta,C}$ and $\vertiii{\cdot}_{n}$ for every $n\in\mathbb{N}$ as in equations \eqref{DefinitionNormAlphaDelta},\eqref{NormN}. Recall that $\vertiii{\cdot}_n$ is UR for all $n\in\mathbb{N}$. The rest of the proof is very similar to the one of the previous theorem. 

    Indeed, to show that $\vertiii{\cdot}_\Omega$ is not UR, it is enough to consider for every $n\in\mathbb{N}$ the points $a_n=\lambda_nx_n+\frac{1-\lambda_n}{2(1+\delta)}h_n$ and $b_n=\lambda_nx_n-\frac{1-\lambda_n}{2(1+\delta)}h_n$ with $\lambda_n=\frac{1}{\vertiii{x_n}_{\alpha_n,\delta,C}}$, and apply Lemma \ref{Lemma2.Norm_A_notUR} to the resulting sequences. 

    Finally, we prove that $\vertiii{\cdot}_\Omega$ is WUR. Let $\{c_k\}_{k\in\mathbb{N}}$ and $\{d_k\}_{k\in\mathbb{N}}$ be two sequences in $S_{\vertiii{\cdot}_\Omega}$ such that $\vertiii{\frac{c_k+d_k}{2}}_{\Omega}\rightarrow 1$. Suppose for the sake of contradiction that $\{d_k-c_k\}_{k\in\mathbb{N}}$ is not weakly null, and, by passing to a subsequence, fix $f_0\in X^*$ and $\varepsilon_0>0$ such that $|\langle f_0,d_k-c_k\rangle| >\varepsilon_0$ for all $k\in\mathbb{N}$. 

    Define for every $n\in\mathbb{N}$ the norm $\vertiii{\cdot}_{\Omega_n}=\max\{\|\cdot\|,\vertiii{\cdot}_n\}$, which is UR by Lemma \ref{Lemma6:Maximumofrotund}. In particular, it is WUR. If there existed $n_0\in\mathbb{N}$ such that $\vertiii{\frac{c_k+d_k}{2}}_\Omega=\vertiii{\frac{c_k+d_k}{2}}_{\Omega_{n_0}}$ for all $k\in\mathbb{N}$, we would reach a contradiction immediately.

    Hence, passing to a further subsequence, we assume that there exists a sequence $\{n_k\}_{k\in\mathbb{N}}$ of natural numbers with $n_k\rightarrow \infty$ such that $\vertiii{\frac{c_k+d_k}{2}}_{\Omega}=\vertiii{\frac{c_k+d_k}{2}}_{n_k}>\left\|\frac{c_k+d_k}{2}\right\|$ for all $k\in\mathbb{N}$. Reasoning exactly as in the proof of the previous theorem, we obtain that $Q_{\|\cdot\|}(P_{\alpha_{n_k}}c_k,P_{\alpha_{n_k}}d_k)\rightarrow 0$, where $P_{\alpha_{n_k}}\colon X\rightarrow \text{ker}(g_{n_k})$ is the linear projection given by $P_{\alpha_{n_k}}x=x-\langle g_{n_k},x\rangle h_{n_k}$. Since $\|\cdot\|$ is UR, Proposition \ref{Prop3.UR_Q_Characterization} shows that $\{P_{\alpha_{n_k}}(d_k-c_k)\}_{k\in\mathbb{N}}$ converges to $0$ in norm. Then, since
    $$d_k-c_k=P_{\alpha_{n_k}}(d_k-c_k)+\langle g_{n_k},d_k-c_k\rangle h_{n_k} $$
    for all $k\in\mathbb{N}$, and the sequence $\{h_{n_k}\}_{k\in\mathbb{N}}$ is weakly null, we obtain that $\{d_k-c_k\}_{k\in\mathbb{N}}$ is weakly null as well. This leads to the contradiction we sought. 

    As in the previous theorems, by taking $\delta$ as small as necessary, it is clear that every norm can be uniformly approximated on bounded sets by norms with this properties.
\end{proof}

\section{$C^\infty$-smooth norm in $c_0$ with non-strictly convex dual norm}

In the final section of this note, we prove the last of the main theorems by constructing a $C^\infty$-smooth norm in $c_0$ whose dual norm is not strictly convex. The process we employ consists in considering first a non-smooth norm in $c_0$ whose dual norm is not strictly convex, and which can be constructed inductively through its finite-dimensional sections. These finite-dimensional sections of the non-smooth ball are $n$-dimensional polyhedra, which we can approximate by $C^\infty$-smooth unit balls in $\ell_\infty^n$. We then use these increasingly better $C^\infty$-smooth approximations to recreate the inductive definition of the original norm with non strictly convex dual norm while preserving $C^\infty$-smoothness.

Note that every equivalent norm in the finite-dimensional space $\ell_\infty^n$ can be uniformly approximated by a $C^\infty$-smooth norm. This is a standard result, which is discussed and improved in \cite{DevFonHaj98}. 

\begin{proof}[Proof of Theorem \ref{Main_Theorem_D}]

    We divide the proof into three parts for the convenience of the reader.

    \vspace{2mm}
    \textbf{1.- Set up and construction of finite-dimensional polyhedra}
    \vspace{2mm}

    Let $\{e_n\}_{n\in\mathbb{N}}$ be the canonical basis of $c_0$, and let $P_n\colon c_0\rightarrow \ell_\infty^n=\text{span}\{e_j\colon 1\leq j\leq n\}$ be the linear projection given by $P_n(x_j)_{j=1}^\infty=(x_j)_{j=1}^n$ for all $(x_j)_{j=1}^\infty\in c_0$ and all $n\in\mathbb{N}$. Consider $0<\delta<1$, and consider a sequence $\{\varepsilon_n\}_{n\in\mathbb{N}}$ of strictly positive numbers with $\varepsilon_1=1$ and such that $\sum_{n=1}^\infty \varepsilon_n\leq\frac{1}{\delta}$. Set $c_n= \sum_{j\leq n}\varepsilon_j$, and choose 
    \begin{align*}
        w_n\in \left(1,\frac{1-\delta c_{n-1}}{1-\delta c_{n-1}-\delta\varepsilon_n}\right)\qquad\text{ and }\qquad h_n= \frac{w_n\delta}{w_n-1},
    \end{align*}
    for every $n\geq 2$, where we select $w_n$ such that 
    \begin{align*}
        &\prod_{n=2}^\infty w_n<\infty\qquad\text{ and }\qquad\sum_{n=2}^\infty \frac{1}{h_n}<\infty.
    \end{align*}
    Consider as well a sequence $\{\eta_n\}_{n\geq 2}$ of strictly positive numbers such that 
    \begin{align*}
        (1+\eta_n)^2\left(\frac{1}{w_n}+\frac{\delta}{2h_n}\right)\leq 1\qquad\text{ and }\qquad (1+\eta_n)^2\left(\frac{h_{n+1}-1}{h_{n+1}-\delta}\right)\leq 1
    \end{align*}
    for every $n\geq 2$. We may assume as well that $\prod_{n=2}^\infty (1+\eta_n)^2\leq \frac{1}{\delta}$, and that $\frac{\delta}{2}\leq \frac{1}{1+\eta_n}$ for all $n\geq 2$.

    Define $\vertiii{\cdot}_{\infty,1}\colon \mathbb{R}\rightarrow \mathbb{R}^+$ by $\vertiii{x}_{\infty,1}= |x|$ for all $x\in \mathbb{R}$.
    
    Inductively, we define for each $n\geq 2$ the norms $\vertiii{\cdot}_{1,n}\colon \ell_\infty^n\rightarrow \mathbb{R}^+$ and $\vertiii{\cdot}_{\infty,n}\colon\ell_\infty^n\rightarrow \mathbb{R}^+$ by
    \begin{equation*}
        \vertiii{x}_{1,n}=\frac{1}{w_n}\vertiii{P_{n-1}x}_{\infty,n-1}+\frac{1}{h_n}|x_n|,\qquad\text{for every }x\in\ell_\infty^n,    
    \end{equation*}
    and
    \begin{equation*}
        \vertiii{x}_{\infty,n}=\max\left\{\vertiii{x}_{1,n},\vertiii{P_{n-1}x}_{\infty,n-1},|x_n|\right\},\qquad\text{for every }x\in\ell_\infty^n.
    \end{equation*}
    Note that $\|x\|_\infty\leq \vertiii{x}_{\infty,n}$ for all $x\in \ell_\infty^n$ and all $n\in\mathbb{N}$. Moreover, it also holds that $\vertiii{x}_{\infty,n}\leq c_n\|x\|_\infty$ for all $n\in\mathbb{N}$ and $x\in\ell_\infty^n$. Indeed, this is trivially true for $n=1$, and assuming it holds for $n-1$ with $n\geq 2$ we obtain that for any $x\in\ell_\infty^n$:
    \begin{align*}
        \vertiii{x}_{\infty,n}&\leq \max\left\{\left(\frac{c_{n-1}}{w_n}+\frac{1}{h_n}\right)\|x\|_\infty,c_{n-1}\|x\|_\infty,\|x\|_\infty\right\}\\       
        &= \max\left\{\left(\frac{c_{n-1}}{w_n}+\frac{w_n-1}{w_n\delta}\right)\|x\|_\infty,c_{n-1}\|x\|_\infty \right\}\\
        &\leq \max\left\{\left(c_{n-1}+\varepsilon_n\right)\|x\|_\infty,c_{n-1}\|x\|_\infty \right\}=c_n\|x\|_\infty.
    \end{align*}
    For the last inequality, we use that $-(1-\delta c_{n-1})\leq -w_n(1-\delta c_{n-1}-\delta\varepsilon_n)$ due to the choice of $w_n$.

    Geometrically, the unit ball of $\vertiii{\cdot}_{\infty,n}$ is an $n$-dimensional polyhedron contained in the unit ball of $\ell_\infty^n$. Figure \ref{fig1_c0} shows the unit ball of $\vertiii{\cdot}_{\infty,2}$ and the geometric implications of the choice of $w_2$ and $h_2$ with respect to $\delta$. 
    \begin{figure}
       \centering
       \includegraphics[scale=0.5]{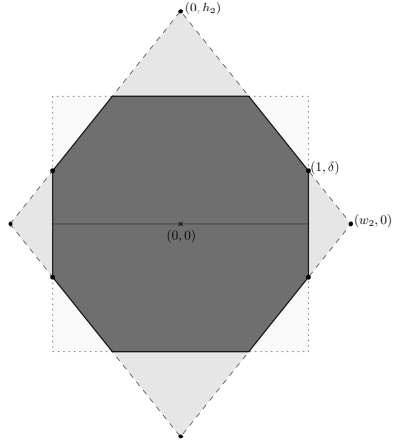}
       \caption{The dark grey region shows the unit ball of $\vertiii{\cdot}_{\infty,2}$. The light gray region is the unit ball of the norm $\vertiii{\cdot}_{1,2}$.}        
        \label{fig1_c0}
    \end{figure}

    \vspace{2mm}
    \textbf{2.- Smooth approximations and definition of the final renorming}
    \vspace{2mm}

    Next, for every $n\geq 2$, we may approximate the norm $\vertiii{\cdot}_{1,n}$ by a $C^\infty$-smooth norm $\vertiii{\cdot}_{\eta_n}\colon \ell_\infty^n\rightarrow \mathbb{R}^+$ such that
    \begin{equation*}
        \vertiii{\cdot}_{1,n}\leq \vertiii{\cdot}_{\eta_n}\leq (1+\eta_n)\vertiii{\cdot}_{1,n}.
    \end{equation*}
    
   By induction, we will define norms $\vertiii{\cdot}_n\colon\ell_\infty^n\rightarrow \mathbb{R}^+$ with the following properties for every $n\in\mathbb{N}$:
    \begin{itemize}

        \item[(i)] $\vertiii{\cdot}_n$ is a $C^\infty$-smooth norm in $\ell_\infty^n$.
        \item[(ii)] If $n\geq 2$, for all $x\in\ell_{\infty}^n$:
        \begin{equation*}
         \vertiii{x}_n\leq (1+\eta_n)\max\left\{\vertiii{x}_{\eta_n},\vertiii{P_{n-1}x}_{n-1},|x_n|\right\}\leq \prod_{j=2}^{n}(1+\eta_j)^2\vertiii{x}_{\infty,n},
        \end{equation*}  
        and
        \begin{equation*}
            \vertiii{x}_n\geq \max\left\{\vertiii{x}_{\eta_n},\vertiii{P_{n-1}x}_{n-1},|x_n|\right\}\geq\vertiii{x}_{\infty,n},
           \end{equation*}  
           for all $1\leq j\leq n$.
        \item[(iii)] If $n\geq 2$, for all $x\in \ell_{\infty}^n$, if $|x_n|\leq\frac{\delta}{2}\|x\|_\infty$, then $\vertiii{x}_n=\vertiii{P_{n-1}x}_{n-1}$.
    \end{itemize} 
    Setting $\vertiii{\cdot}_1=\vertiii{\cdot}_{\infty,1}=|\cdot|$, the previous conditions are satisfied trivially for $n=1$. Suppose that $\vertiii{\cdot}_{n-1}\colon \ell_{\infty}^{n-1}\rightarrow \mathbb{R}^+$ has been defined with the desired properties for a fixed $n\geq 2$.

    Now, let $\phi_n\colon \mathbb{R}\rightarrow \mathbb{R}^+$ be a real valued, $C^\infty$-smooth and convex function such that $\phi_n(t)=0$ for all $t\in\left[0,\frac{1}{1+\eta_n}\right]$, $\phi_n(1)=1$, and $\phi_n(t)>1$ for all $t\in (1,+\infty)$. The map $\nu_n\colon \ell_\infty^n\rightarrow \mathbb{R}^+$ given by 
    $$\nu_n(x)=\phi_n(\vertiii{x}_{\eta_n})+\phi_n\left(\vertiii{P_{n-1}x}_{n-1}\right)+\phi_n\left(|x_n|\right)$$
    is a $C^\infty$-smooth function such that the set $B_n=\{x\in\ell_\infty^n\colon \nu_n(x)\leq 1\}$ is bounded in $\ell_\infty^n$. Hence, by Corollary 1.1.23 in \cite{RuPhd}, the set $B_n$ is the unit ball of a $C^\infty$-smooth norm $\vertiii{\cdot}_n\colon \ell_\infty^n\rightarrow \mathbb{R}^+$. Notice that, given $x\in\ell_\infty^n$, $\vertiii{x}_n=1$ if and only if $\nu_n(x)=1$. Let us show that properties $(ii)$ and $(iii)$ are satisfied by this norm.  

    We prove that $(ii)$ holds for $x\in\ell_\infty^n$ with $\vertiii{x}_n=1$. For the first part, notice that if $\nu_n(x)=1$, then one of $\phi_n(\vertiii{x}_{\eta_n})$, $\phi_n\left(\vertiii{P_{n-1}x}_{n-1}\right)$ or $\phi_n\left(|x_n|\right)$ must be strictly positive. The first inequality now follows since $\phi_n$ vanishes in the interval $\left[0,\frac{1}{1+\eta_n}\right]$. 
    
    The second inequality of the first statement holds simply by definition of $\vertiii{\cdot}_{\infty,n}$ if $n=2$; while if $n>2$ we also need to use the inductive hypothesis.

    For the second statement, observe that if $\nu_n(x)=1$, then 
    $$\max\left\{\phi_n(\vertiii{x}_{\eta_n}),\phi_n\left(\vertiii{P_{n-1}x}_{n-1}\right),\phi_n\left(|x_n|\right)\right\}\leq 1,$$ 
    from which the first inequality follows since $\phi_n(t)>1$ for all $t\in(1,+\infty)$. The second inequality is a simple application of the inductive hypothesis if $n>2$. Note that the case $n=2$ should be argued separately, when the conclusion holds by the definitions of both $\vertiii{\cdot}_{\infty,2}$ and $\vertiii{\cdot}_{\eta_2}$, and the fact that $\vertiii{\cdot}_1=\vertiii{\cdot}_{\infty,1}$.

    Before proving that $(iii)$ holds as well, notice that since $\|\cdot\|_\infty\leq\vertiii{\cdot}_{\infty,n} $, the last inequality of condition $(ii)$ implies in particular that $\|x\|_\infty\leq \vertiii{x}_n$ for every $x\in\ell_\infty^n$. 

    Now, to prove property $(iii)$ we may fix $x\in\ell_\infty^n$ with $|x_n|\leq \delta\|x\|_\infty$ and assume that $\vertiii{x}_n=1$. By the previous observation, we have that $\|x\|_\infty\leq 1$, and thus $|x_n|\leq \frac{\delta}{2}$. Since $\frac{\delta}{2}\leq \frac{1}{1+\eta_n}$, we obtain directly that $\phi_n\left(|x_n|\right)=0$.

    We have as well that $\vertiii{P_{n-1}x}_{\infty,n-1}\leq 1$ by the last part of the already proven property $(ii)$, and thus we get the following estimate:
    \begin{align*}
        \vertiii{x}_{\eta_n}&\leq (1+\eta_n)\vertiii{x}_{1,n}=(1+\eta_n)\left(\frac{\vertiii{P_{n-1}x}_{\infty,n-1}}{w_n}+\frac{|x_n|}{h_n}\right)\\
        &\leq (1+\eta_n)\left(\frac{1}{w_n}+\frac{\delta}{2h_n}\right)=\frac{1}{1+\eta_n},
    \end{align*}
    since we have chosen $\eta_n$ such that the last inequality holds. Hence, we also get that $\phi_n(\vertiii{x}_{\eta_n})=0$. Since $\nu_n(x)=1$, we necessarily have that $\phi_n\left(\vertiii{P_{n-1}x}_{n-1}\right)=1$, and by convexity of $\phi_n$ we obtain that $\vertiii{P_{n-1}x}_{n-1}=1$, which proves $(iii)$.


    Once the induction is finished, we define the final renorming $\vertiii{\cdot}\colon c_0\rightarrow\mathbb{R}^+$ by 
    $$\vertiii{x}=\sup_{n\in\mathbb{N}}\left\{\vertiii{P_nx}_n\right\}$$ 
    for all $x\in c_0$. Property $(ii)$ in the induction and the fact that $\|\cdot\|_\infty\leq\vertiii{\cdot}_{\infty,n}\leq \frac{1}{\delta}\|\cdot\|_\infty$ for all $n\in\mathbb{N}$ show that
    $$ \|\cdot\|_\infty\leq \vertiii{\cdot}\leq \frac{1}{\delta^2}\|\cdot\|_\infty.$$
    
    Using property $(iii)$ we also have that given $x_0\in c_0$ with $x_0\neq 0$, there exists an open neighbourhood $U$ of $x_0$ and $n_0\in\mathbb{N}$ such that $\vertiii{x}=\vertiii{P_{n_0}x}_{n_0}$ for all $x\in U$. This implies that the norm is LFC and $C^\infty$-smooth.
    
    \vspace{2mm}
    \textbf{3.- Non-strictly convex dual ball}
    \vspace{2mm}

    It only remains to prove that the dual norm of $\vertiii{\cdot}$ is not strictly convex. Define
    \begin{align*}
        f &= \left(\prod_{j>1}^\infty\frac{1}{w_j},\frac{1}{h_2}\prod_{j>2}^\infty\frac{1}{w_j},\dots,\frac{1}{h_n}\prod_{j>n}^\infty\frac{1}{w_j},\dots\right)\in\ell_1\\
        g &= \left(0,\prod_{j>2}^\infty\frac{1}{w_j},\frac{1}{h_3}\prod_{j>3}^\infty\frac{1}{w_j},\dots,\frac{1}{h_n}\prod_{j>n}^\infty\frac{1}{w_j},\dots\right)\in\ell_1.
    \end{align*}
    We will show that the segment $[f,g]$ is contained in the sphere $S_{\vertiii{\cdot}^*}$.

    Given $x\in c_0$, using the fact that $\vertiii{P_1x}_{\infty,1}=|x_1|$ and the definition of $\vertiii{\cdot}_{\infty,n}$, it is straightforward to show inductively that 
    $$\prod_{j>1}^n\frac{1}{w_j}|x_1|+\sum_{i=2}^n \frac{1}{h_i}\prod_{j>i}^{n}\frac{1}{w_j}|x_i|\leq \vertiii{P_nx}_{\infty,n}$$
    for all $n\geq 2$. Since $w_n> 1$ for all $n\geq 2$ and $\vertiii{P_nx}_{\infty,n}\leq \vertiii{x}$ for all $n\in\mathbb{N}$, we obtain that 
    $$ \prod_{j>1}^\infty\frac{1}{w_j}|x_1|+\sum_{i=2}^\infty \frac{1}{h_i}\prod_{j>i}^{\infty}\frac{1}{w_j}|x_i|\leq \vertiii{x}.$$
    
    This shows that $|\langle f,x\rangle|\leq \vertiii{x}$ and $|\langle g,x\rangle |\leq \vertiii{x}$, which implies that $\vertiii{f}^*\leq 1$ and $\vertiii{g}^*\leq 1$. 
    
    Consider now for every $n\geq2$ the point $z_n\in c_0$ given by
    $$ z_n = \left(\prod_{j>1}^n \frac{h_j-1}{h_j-\delta},\dots,\prod_{j>n-1}^n \frac{h_j-1}{h_j-\delta},1,0,\dots\right)\in\text{span}\{e_1,\dots,e_n\}.$$
    Setting $z_1=e_1$, we have that $P_{n-1}z_n =\frac{h_n-1}{h_n-\delta}P_{n-1}z_{n-1}$ for all $n\geq 2$. We start by showing inductively that $\vertiii{P_nz_n}_{\infty,n}=\vertiii{P_nz_n}_{1,n}=1$. Indeed, this is trivially true for $n=1$, and given $n\geq 2$ and assuming it holds for $n-1$ we obtain that
    \begin{align*}
        \vertiii{P_nz_n}_{1,n}&= \frac{1}{w_n}\vertiii{P_{n-1}z_n}_{\infty,n-1}+\frac{1}{h_n}|z_n| \\
        &=\frac{1}{w_n}\frac{h_n-1}{h_n-\delta}+\frac{1}{h_n}=1.
    \end{align*}
    Now the conclusion follows by definition of $\vertiii{\cdot}_{\infty,n}$ and a further application of the inductive hypothesis.
    
    Thanks to the previous fact, and again inductively, we can show that $\vertiii{z_n}\leq (1+\eta_n)^2$ for all $n\geq 2$. In order to do this, notice that if $\vertiii{z_{n-1}}\leq (1+\eta_{n-1})^2$ for some $n>2$, using that $(1+\eta_{n-1})^2\left(\frac{h_n-1}{h_n-\delta}\right)\leq 1$ and property $(ii)$ of the definition of $\vertiii{\cdot}_n$, we get that
    \begin{align*}
        \vertiii{z_n}&=\vertiii{P_nz_n}_n=(1+\eta_n)\max\left\{\vertiii{P_nz_n}_{\eta_n},\vertiii{P_{n-1}z_n},|z_n|\right\}\\
        &\leq (1+\eta_n)\max\left\{(1+\eta_n)\vertiii{P_nz_n}_{1,n},(1+\eta_{n-1})^2\left(\frac{h_n-1}{h_n-\delta}\right)\right\}= (1+\eta_n)^2.
    \end{align*}
    Finally, we show that $\langle f, z_n\rangle = \langle g,z_n\rangle =\prod_{j>n}^\infty\frac{1}{w_j}$ for $n\geq 2$. We start by showing the base case $n=2$. For the functional $g$, the equality is straightforward from the definition. For $f$, simply observe that:
    \begin{align*}
        \langle f,z_2\rangle &= \frac{h_2-1}{h_2-\delta}\prod_{j>1}\frac{1}{w_j}+\frac{1}{h_2}\prod_{j>2}\frac{1}{w_j}\\
        &=\prod_{j>2}\frac{1}{w_j}\left(\frac{h_2-1}{h_2-\delta}\frac{1}{w_2}+\frac{1}{h_2}\right)=\prod_{j>2}\frac{1}{w_j}.
    \end{align*}

    Assume now the conclusion holds for $n-1$ for some $n>2$. Then, for $\varphi\in \{f,g\}$, and writing $\varphi_{|n}=(\varphi_i)_{i=1}^n\in(\ell_\infty^n)^*$ to denote the projection of $\varphi$ onto its first $n$ coordinates, we obtain:
    \begin{align*}    
        \langle \varphi,z_n\rangle &=\langle \varphi_{|n},P_nz_n\rangle =\langle \varphi_{|(n-1)}, P_{n-1}z_n\rangle +\frac{1}{h_n}\prod_{j>n}^\infty \frac{1}{w_j}\\
        &=\frac{h_n-1}{h_n-\delta}\langle \varphi_{|(n-1)},P_{n-1}z_{n-1}\rangle +\frac{1}{h_n}\prod_{j>n}^\infty \frac{1}{w_j}\\
        &=\prod_{j>n}^\infty\frac{1}{w_j}\left(\frac{h_n-1}{h_n-\delta}\frac{1}{w_n}+\frac{1}{h_n}\right)=\prod_{j>n}^\infty\frac{1}{w_j}.
    \end{align*}
    This implies that $\left\langle \frac{f+g}{2},z_n\right\rangle=\prod_{j>n}^\infty\frac{1}{w_j}$ as well. Since we have found a sequence $\{z_n\}_{n\in\mathbb{N}}\subset c_0$ such that $\vertiii{z_n}\rightarrow 1$ and $\left\langle \frac{f+g}{2},z_n\right\rangle\rightarrow 1$, we get that $\vertiii{\frac{f+g}{2}}^*\geq 1$. This finishes the proof, as we have shown already that the norm of $f$ and $g$, and hence the norm of $\frac{f+g}{2}$, is less than $1$.
\end{proof}

\section*{Acknowledgements}
This research was supported by GA23-04776S and by the project SGS22/053/OHK3/1T/13. The second author's research has been supported by PAID-01-19 and by grant PID2021-122126NB-C33 funded by MCIN/AEI/10.13039/501100011033 and by “ERDF A way of making Europe”.

\printbibliography
\end{document}